\documentclass[final]{siamltex}
\usepackage{amsfonts,amsmath,amssymb}
\usepackage{mathrsfs,mathtools}
\usepackage{enumerate}
\usepackage{xspace}
\usepackage{hyperref}
\usepackage{esint}
\usepackage{graphicx}
\DeclareMathAlphabet{\mathpzc}{OT1}{pzc}{m}{it}

\usepackage{color}

\newcommand{\ue}{\textup{\textsf{u}}}
\newcommand{\Lp}{{L^p(\Omega)}}
\newcommand{\bLp}{{\mathbf{L}^p(\Omega)}}
\newcommand{\BMO}{{\textup{\textsf{BMO}}(\Omega)}}
\newcommand{\bmo}{{\textup{\textsf{BMO}}}}
\newcommand{\vare}{{\varepsilon}}
\newcommand{\DIV}{\nabla\!{\cdot}}   
\newcommand{\GRAD}{\nabla}           
\newcommand{\fe}{\textup{\textsf{f}}}
\newcommand{\Real}{\mathbb R}

\newcommand{\polM}{{\mathbb M}}
\newcommand{\polN}{{\mathbb N}}

\newcommand{\polP}{{\mathbb P}}
\newcommand{\polQ}{{\mathbb Q}}

\newcommand{\polV}{{\mathbb V}}

\newcommand{\polX}{{\mathbb X}}

\newcommand{\ie}{i.e.,\@\xspace}

\def\scl{\left\langle}
\def\scr{\right\rangle}
\def\Ldeuxd{{{ \bf L}^2   (\Omega)}}
\newcommand{\diff}{\, \mbox{\rm d}}
\newcommand{\calM}{{\mathcal M}}
\def\Ldeux{{{  L}^2   (\Omega)}}
\newcommand{\calT}{{\mathcal T}}
\def\tr{^\intercal}
\newcommand{\calL}{{\mathcal L}}
\newcommand{\calR}{{\mathcal R}}
\newcommand{\calS}{{\mathcal S}}
\newcommand{\hu}{{\widehat{u}}}
\newcommand{\hp}{{\hat{p}}}
\newcommand{\hq}{{\hat{q}}}

\newcommand{\bell}{{\boldsymbol \ell}}
\newcommand{\frakh}{{\mathfrak h}}

\newcommand{\bL}{{\bf L}}

\newcommand{\bg}{{\bf g}}
\newcommand{\bh}{{\bf h}}

\newcommand{\bq}{{\bf q}}

\newcommand{\bs}{{\bf s}}

\newcommand{\bw}{{\bf w}}

\newtheorem{remark}[theorem]{Remark}
\numberwithin{equation}{section}

\title{Approximation of elliptic equations with $\bmo$ coefficients\thanks{AJS is partially supported by NSF grant DMS-1418784.}}

\author{Harbir Antil\thanks{Department of Mathematical Sciences, George Mason University, Fairfax, VA 222030, USA \texttt{hantil@gmu.edu}},
\and
Abner J.~Salgado\thanks{Department of Mathematics, University of Tennessee, Knoxville, TN 37996, USA. \texttt{asalgad1@utk.edu}}}

\date{draft of \today.}

\begin{document}

\maketitle
\begin{abstract}
We study solution techniques for elliptic equations in divergence form, where the coefficients are only of bounded mean oscillation ($\bmo$). For $|p-2|<\vare$ and a right hand side in $W^{-1}_p$ we show convergence of a finite element scheme, where $\vare$ depends on the oscillation of the coefficients.
\end{abstract}



\section{Introduction}
\label{sec:introduccion}

In this work we are interested in the analysis of the convergence properties of a standard finite element scheme for the approximation of a linear elliptic boundary value problem in divergence form
\begin{equation}
\label{eq:bvp}
  \DIV \left( A(x) \GRAD \ue \right) = \DIV \fe \quad \text{in } \Omega, \qquad \ue_{|\partial\Omega} = 0,
\end{equation}
where $\Omega \subset \Real^d$ with $d\geq1$ is a bounded domain with Lipschitz boundary and $\fe \in \bLp$ for some $p \in (1,\infty)$. The main originality and source of difficulty here is that we only assume that the matrix $A:\Omega \rightarrow \polM^d$ is symmetric, positive definite --- \ie that there is a constant $\alpha>0$ such that, for almost every $x \in \Omega$,
\begin{equation}
\label{eq:Apositive}
  \scl A(x) \xi, \xi \scr \geq \alpha |\xi|^2 \quad \forall \xi \in \Real^d,
\end{equation}
and that its entries are in $\BMO$ (see \S\ref{sec:notation} for notation). Notice that with this very restricted regularity on $A$, having $\GRAD w \in \Ldeuxd$ is not enough to even guarantee that the energy functional
\[
  w \mapsto \int_\Omega \scl A(x) \GRAD w(x), \GRAD w(x) \scr \diff x
\]
is bounded, see \S\ref{sub:bmo} for details.

The existence and uniqueness of solutions to \eqref{eq:bvp} for the case of bounded coefficients --- $A_{i,j} \in L^\infty(\Omega)$, for $i,j = \overline{1,d}$ --- and datum $\fe \in \Ldeuxd$ is standard in the literature. Notice however, that very little has been said concerning the approximation of solutions without further assumptions on the coefficients. In fact, as recent work shows \cite{MR3129757}, even a bounded but discontinuous coefficient can be the source of difficulties and a nonstandard treatment might be necessary to assert convergence. Assuming bounded coefficients, the classical work of N.G.~Meyers  \cite{MR0159110} shows that there is $\vare>0$ such that if $p\in (2-\vare, 2+\vare)$ and $\fe \in \bLp$, then $\GRAD \ue \in \bLp$. In this setting, the convergence of finite element schemes in $W^1(\Lp)$ is a well studied topic in the literature \cite[\S8.6]{MR2373954}. It is important to remark that the parameter $\vare$ for which these estimates are valid strongly depends on the ellipticity constants of the coefficient, and that this can be made as small as desired by increasing the value of $M$ on the upper bound
\[
  \scl A(x) \xi, \xi \scr \leq M |\xi|^2 \quad \forall \xi \in \Real^d,
\]
which we do not assume (in our case $M=\infty$). The reader is referred to \cite[\S5]{MR0159110} for an example.

On the other hand, removing the assumption of boundedness on the coefficients $A$ seems to be rather new. In particular, understanding what are the minimal requirements for \eqref{eq:bvp} to be well posed and a Meyers-type $\bLp$ estimate have only been recently studied by B.~Stroffolini in \cite{MR1837269} and S.S.~Byun in \cite{MR2110431}. To set our work in context, let us briefly describe these results. Given $\fe \in \bLp$ a variational solution of \eqref{eq:bvp} is a function $\ue \in W^1_0(\Lp)$ such that
\[
  \int_\Omega \scl A(x) \GRAD \ue , \GRAD \zeta \scr \diff x = \int_\Omega \scl \fe, \GRAD \zeta \scr \quad \forall \zeta \in C^\infty_0(\Omega).
\]
Then \cite[Theorem~1.1]{MR1837269} and \cite[Theorem~1.5]{MR2110431} state that there exists $\vare>0$ that depends on the $\BMO$ norm of the coefficient $A$, such that if $|p-2|<\vare$ then problem \eqref{eq:bvp} is well posed. To the best of our knowledge, no analysis of the approximation properties of finite elements in this setting has been made, and it is the purpose of this work to fill this gap. We will show that for a similar range of $p$, \ie that depends on the oscillation of $A$, a standard finite element scheme converges to the solution of \eqref{eq:bvp}.

To obtain convergence of our finite element scheme we will approximate the coefficient $A$ with piecewise constants, while the solution will be approximated with piecewise linears. In this setting the flux $A_h \GRAD u_h$ is piecewise constant and so we present a discrete Hodge decomposition of piecewise constant vector fields. The importance of this decomposition is that it allows us to obtain the key a priori estimate of Theorem~\ref{thm:apriori}, which is then used to pass to the limit and show that the weak limit solves \eqref{eq:bvp}. Strong convergence is then established with the help of inf-sup theory.

Our presentation is organized as follows: In section~\ref{sec:notation} we introduce the notation and main assumptions we will operate with. In particular, in \S\ref{sub:bmo} we review the most important properties of functions in the John Nirenberg space $\BMO$. The discrete Hodge decomposition of piecewise constant vector fields and some of its consequences are discussed in section~\ref{sec:Hodge}. The core of our work is section~\ref{sec:MKE} where we introduce our finite element scheme and show that it converges to the solution of \eqref{eq:bvp}. Finally, the possibility of establishing rates of convergence is discussed in section~\ref{sec:convrates}.

\section{Notation and preliminaries}
\label{sec:notation}
In this work $\Omega \subset \Real^d$, $d\geq1$, is a convex domain with polyhedral boundary. We denote by $\scl\cdot,\cdot\scr$ the Euclidean inner product in $\Real^d$ and, for $x \in \Real^d$, $|x|^2 = \scl x, x \scr$. For $B \subset \Real^d$ we denote by $|B|$ its Lebesgue measure. If $X(\Omega)$ is a Banach function space over $\Omega$, we denote by $\|\cdot\|_X$ its norm. Spaces of vector valued functions will be denoted by boldface characters. By $\Lp$ with $p \in [1,\infty]$ we denote the space of functions that are Lebesgue integrable with exponent $p$. By $W^k(\Lp)$ we denote the classical Sobolev space of functions whose distributional derivatives of up to $k$-th order are in $\Lp$. The closure of $C^\infty_0(\Omega)$ in $W^k(\Lp)$ is denoted by $W^k_0(\Lp)$. The conjugate exponent to $p$ will be denoted by $q$, \ie $\tfrac1p+\tfrac1q = 1$.

In what follows we will denote nonessential constants by $C$.

\subsection{Functions of bounded mean oscillation and their properties}
\label{sub:bmo}

In 1961 F.~John and L.~Nirenberg \cite{MR0131498} introduced the space of functions of bounded mean oscillation and proved many of its fundamental properties. For completeness we recall some basic definitions and properties, we refer to \cite{MR0131498,MR1658640,MR621018} and \cite[Chapter 6]{MR1800316} for details. For a function $w \in L^1(\Omega)$ and a cube $Q \subset \Omega$ we denote by $w_Q$ the average of $w$ on $Q$:
\[
  w_Q = \fint_Q w(x) \diff x = \frac1{|Q|} \int_Q w(x) \diff x.
\]
Define the sharp maximal function by
\begin{equation}
\label{eq:sharpmax}
  \calM^\sharp w(x) = \sup_{Q \ni x} \fint_Q |w(z) - w_Q | \diff z.
\end{equation}
We say that a function $w$ has bounded mean oscillation if $\calM^\sharp w$ is bounded, \ie
\[
  \BMO = \left\{ w \in L^1(\Omega): \calM^\sharp w \in L^\infty(\Omega) \right\}.
\]
We define a seminorm on $\BMO$ by
\[
  | w |_{\bmo} = \| \calM^\sharp w \|_{L^\infty}.
\]
Notice that this is not a norm since a function that is almost everywhere constant has zero oscillation. It can be shown, however, that these are the only functions with this property. Clearly $L^\infty(\Omega)\subset \BMO$, but there are also unbounded $\bmo$ functions. A classical example is
\begin{equation}
  w(x)=\begin{dcases}
         \log\left( \frac1{|x|} \right) & |x|<1, \\ 0 & |x|\geq 1.
       \end{dcases}
\label{eq:bmonotlinf}
\end{equation}

It is easy to see, in addition, that the Hardy Littlewood maximal function
\[
  \calM w(x) = \sup_{Q \ni x} \fint_Q |w(z)| \diff z
\]
and the sharp maximal function  are related by the point wise inequality
\[
  \calM^\sharp w(x) \leq C \calM w(x).
\]
While the reverse inequality is not true, the celebrated Fefferman-Stein inequality states that if $w \in \Lp$ with $p \in (1,\infty)$, then
\begin{equation}
  \| \calM w \|_{L^p} \leq C \| \calM^\sharp w \|_{L^p}.
\label{eq:FefStein}
\end{equation}
A fundamental property of $\BMO$ functions is the John-Nirenberg inequality which, simply put, gives an exponential decay to the distribution of a $\BMO$ function: If $w \in \BMO$, then there are constants $C_1,C_2>0$ such that for any $\lambda > 0$ and $Q \subset \Omega$
\[
  \left| \left\{ x \in Q: | w(x) - w_Q | > \lambda \right\} \right| \leq C_1 |Q| \exp\left( - \frac{C_2 \lambda}{|w|_{\bmo}} \right).
\]
This inequality has many useful consequences. In particular, if $w \in \BMO$ then there is $\lambda > 0$ such that for any cube $Q$
\[
  \fint_Q e^{\lambda | w(x) - w_Q|} \diff x < \infty.
\]
In addition, we obtain that $\BMO \subset L^r(\Omega)$ for all $r<\infty$.

Notice that, although the John-Nirenberg inequality asserts that $\BMO \subset \cap_{r\geq1} L^r(\Omega)$, example \eqref{eq:bmonotlinf} shows that $\BMO \not \subset L^\infty(\Omega)$. This is one of the many reasons why this space has become so important in harmonic analysis. It is a space that sits between $L^\infty(\Omega)$ and $\cap_{r\geq1} L^r(\Omega)$ and that in many aspects, like operator interpolation and duality, can be used to replace $L^\infty(\Omega)$.

Another implication of the inclusions $L^\infty(\Omega) \subsetneq \BMO \subsetneq \cap_{r\geq1} L^r(\Omega)$ and the main source of difficulty and originality in this work is that given $a \in \BMO$ and $w \in \Ldeux$, the integral
\[
  \int_\Omega a(x) |w(x)|^r \diff x
\]
is a priori bounded only if $r<2$. This bears an important consequence with respect to the analysis and approximation of solutions to \eqref{eq:bvp}: Standard techniques do not work! Even if $\fe \in \Ldeuxd$, where for a bounded coefficient well-known energy arguments would provide for a satisfactory analysis of \eqref{eq:bvp} and its approximations, novel techniques are necessary to assert that \eqref{eq:bvp} is well posed (\cite{MR1837269,MR2110431}), let alone to provide a convergent numerical scheme.

\subsection{Finite elements}
\label{sub:fem}

For $h>0$ we introduce a triangulation $\calT_h = \{ K \}$ of $\Omega$ consisting of simplices, which we assume is conforming and shape regular \cite{MR0520174,MR2050138}. Over each triangulation $\calT_h$ we define the finite element space
\[
  \polV_h = \left\{ w_h \in C(\bar\Omega): w_{h|K} \in \polP_1 \right\},
\]
of continuous piecewise linear functions. To account for boundary conditions we define
\[
  \polX_h = \left\{ w_h \in \polV_h: w_{h|\partial\Omega} = 0 \right\}.
\]
We will also make use of a space of vector valued piecewise constant functions
\[
  \polQ_h = \left\{ \bq_h \in \bL^\infty( \Omega ) : \bq_{h|K} \in \Real^d \right\}.
\]

We do not assume that the triangulations $\calT_h$ are quasi uniform, instead we will assume that they are such that allow a best approximation property in $W^1(L^\infty(\Omega))$. In other words, if $w \in W^1_0(L^\infty(\Omega))$ and $w_h \in \polX_h$ solves
\[
  \int_\Omega \scl \GRAD w_h, \GRAD \zeta_h \scr \diff x = \int_\Omega \scl \GRAD w, \GRAD \zeta_h \scr \diff x \quad \forall \zeta_h \in \polX_h,
\]
then we have
\[
  \| \GRAD (w - w_h ) \|_{\bL^\infty} \leq C \min_{\zeta_h \in \polX_h} \| \GRAD( w - \zeta_h ) \|_{\bL^\infty}.
\]
More importantly, this implies
\begin{equation}
  \| \GRAD w_h \|_{\bL^\infty} \leq C \| \GRAD w \|_{\bL^\infty}.
\label{eq:W1inftystab}
\end{equation}
Notice also that interpolation between a trivial energy estimate and \eqref{eq:W1inftystab} yields
\begin{equation}
  \| \GRAD w_h \|_{\bL^p} \leq C \| \GRAD w \|_{\bL^p}
\label{eq:W1pstab}
\end{equation}
for every $p \in [2,\infty]$. Duality yields the same estimate for $p\in (1,2)$. Reference \cite{MR2869035} studies conditions that are much more general than quasi uniformity but still allow for an estimate like \eqref{eq:W1inftystab}. The Appendix of \cite{MR2754849} develops and analyzes an optimal algorithm for the construction of such meshes. In 
simple terms these meshes are such that if $\frakh_{\calT_h}$ is a smooth enough function such that for every $K \in \calT_h$ we have $\frakh_{\calT_h|K} \approx |K|^{1/n}$ --- \ie the local mesh size --- then $\|\GRAD \frakh_{\calT_h}\|_{\bL^\infty}$ is small enough.

As our meshes are not assumed quasi uniform there is no attached mesh size to them and, thus, the discretization parameter $h>0$ does not have any real meaning. Nevertheless, we will stick to standard notation and denote by $h \to 0$ the result of
\[
  \bigcup_{h>0} \polV_h, \qquad \bigcup_{h>0} \polX_h, \qquad \bigcup_{h>0} \polQ_h,
\]
which we assume dense in $W^1(\Lp)$, $W^1_0(\Lp)$ and $\bL^r(\Omega)$, $r \in [1,\infty)$, respectively. Moreover, we assume that the family of triangulations $\{\calT_h\}_{h>0}$ is constructed in such a way that for every $x \in \Omega$ there is a family of simplices that shrinks nicely \cite[\S 3.4]{MR1681462} to $x$. This process can be rigorously described using nets, but we shall not dwell on these technicalities.

When treating the discretization of \eqref{eq:bvp} it will become necessary to compute a variant of the Hardy Littlewood maximal function over $\calT_h$. Namely, if $\calT_h = \{K\}$, then
\[
  \calM_{\calT_h} w (x) = \sup_{K \ni x} \fint_K |w(z)| \diff z.
\]
Notice that for almost every $x\in \Omega$, there is a unique $K \in \calT_h$ such that $x \in K$. The remaining points are those that lie on faces or edges of the mesh and for these there is only a finite number of cells that contain it. Thus, $\calM_{\calT_h} w$ is everywhere defined for any $w \in L^1(\Omega)$.
The following result states that $\calM_{\calT_h}$ is point wise controlled by $\calM$.

\begin{lemma}[relation between maximal functions]
\label{lem:maximals}
If $\calT_h$ is shape regular, then there is a constant $C$ independent of $h$ such that
\[
  \calM_{\calT_h} w (x) \leq C \calM w(x)
\]
for every $w \in L^1(\Omega)$.
\end{lemma}
\begin{proof}
This is an easy consequence of shape regularity. If $x \in K$, there is a cube $Q$ such that $x \in K \subset Q$. Moreover, by shape regularity, the ratio $|Q|/|K|$ is uniformly bounded. This implies
\[
  \fint_K |w(z)| \diff z = \frac{|Q|}{|Q|} \frac1{|K|} \int_K |w(z)| \diff z \leq C \fint_Q |w(z)|\diff z,
\]
where the constant in the last inequality is independent of $K$. Taking suprema yields the result.
\end{proof}

We discretize the coefficient $A$ by a piecewise constant approximation $A_h$, such that every column $A_{h,i} \in \polQ_h$, $i=\overline{1,d}$, as follows
\[
  A_{h|K} = \fint_K A(z) \diff z.
\]

\begin{lemma}[approximation of $\bmo$ coefficients]
\label{lem:approxbmo}
If $A:\Omega \rightarrow \polM^d$ is symmetric, positive definite in the sense of \eqref{eq:Apositive} and its entries are in $\BMO$, then $A_{h,ij} \in L^\infty(\Omega)$ for $i,j = \overline{1,d}$. Moreover, the coercivity constant of $A$ is that of $A_h$. Finally, as $h\to 0$, we have that $\| A - A_h\|_{L^r} \to 0$ for $r \in [1,\infty)$.
\end{lemma}
\begin{proof}
The fact that the entries of $A_h$ are in $L^\infty(\Omega)$ is evident.

To show that the coercivity constants coincide consider $\xi \in \Real^d$ and $K \in \calT_h$
\[
  \scl A_{h|K} \xi, \xi \scr = \scl \fint_K A(z) \diff z\ \xi, \xi \scr = \fint_K \scl A(z) \xi, \xi \scr \diff z \geq \alpha \fint_K |\xi|^2 \diff z 
  = \alpha |\xi|^2 .
\]

The convergence in $L^r(\Omega)$ for any $r \in [1,\infty)$ is obtained as a consequence of the fact that the simplices shrink nicely to points, for this implies that $A_h(x) \to A(x)$ for a.e.~$x \in \Omega$. Using Lemma~\ref{lem:maximals}
\[
  |A_h| = \left| \fint_K |A(z)| \diff z \right| \leq \calM_{\calT_h} A \leq C \calM A.
\]
When $r>1$, $A \in L^r(\Omega)$ together with the fact that the Hardy Littlewood maximal function is bounded in $L^r(\Omega)$ yield $\calM A \in L^r(\Omega)$. Dominated convergence implies that $A_h \to A$ in $L^r(\Omega)$. Convergence in $L^1(\Omega)$ follows from the convergence in, say, $\Ldeux$ and the fact that $\Omega$ is bounded.
\end{proof}

\section{A discrete Hodge decomposition}
\label{sec:Hodge}

The classical Hodge decomposition states that if $\bw \in \bL^r(\Omega)$ with $r \in (1,\infty)$, then it can be uniquely decomposed as $\bw = \GRAD \phi + \bh$, where $\phi \in W^1_0(L^r(\Omega))$ and $\bh \in \bL^r(\Omega)$ is weakly divergence free. Here we present a discrete variant of this result and some applications.

\begin{theorem}[discrete Hodge decomposition]
\label{thm:discHodge}
Let $r \in (1,\infty)$. If $\bs_h \in \polQ_h$, then there are unique $\varphi_h \in \polX_h$ and $\bg_h \in \polQ_h$ such that
\[
  \bs_h = \GRAD \varphi_h + \bg_h, \qquad \int_\Omega \scl \bg_h, \GRAD \zeta_h \scr \diff x = 0 \quad \forall \zeta_h \in \polX_h.
\]
Moreover,
\[
  \| \GRAD \varphi_h \|_{\bL^r} + \| \bg_h \|_{\bL^r} \leq C \| \bs_h \|_{\bL^r},
\]
for a constant that is independent of $\bs_h$ and $h$.
\end{theorem}
\begin{proof}
We split the proof in two steps:

\begin{enumerate}[$\bullet$]
  \item \emph{Existence:} We define $\varphi_h \in \polX_h$ as the solution of
  \[
    \int_\Omega \scl \GRAD \varphi_h, \GRAD \zeta_h \scr \diff x = \int_\Omega \scl \bs_h ,\GRAD \zeta_h \scr \diff x \quad \forall \zeta_h \in \polX_h.
  \]
  By \eqref{eq:W1pstab} we have
  \[
    \| \GRAD \varphi_h \|_{\bL^r} \leq C \| \bs_h \|_{\bL^r}.
  \]
  Define $\bg_h = \bs_h - \GRAD \varphi_h$. By construction
  \[
    \int_\Omega \scl \bg_h , \GRAD \zeta_h \scr \diff x = 0 \quad \forall \zeta_h \in \polX_h, \qquad \| \bg_h \|_{\bL^r} \leq C \| \bs_h \|_{\bL^r}.
  \]
  
  \item \emph{Uniqueness:} Assume that, in addition, $\bs_h = \GRAD \tilde\varphi_h + \tilde\bg_h$, \ie $\GRAD(\varphi_h - \tilde\varphi_h ) + (\bg_h - \tilde\bg_h) = 0$. Multiply this last identity by $\GRAD(\varphi_h - \tilde\varphi_h )$ and integrate to conclude $\varphi_h = \tilde\varphi_h$. This in turn implies that $\bg_h = \tilde \bg_h$.
\end{enumerate}

Notice that the mesh restriction and, more importantly, its corollary \eqref{eq:W1pstab} is fundamental to obtain this result.
\end{proof}

The main application of this discrete Hodge decomposition is in obtaining a priori bounds for a finite element approximation of \eqref{eq:bvp}. The first corollary requires the following nonlinear interpolation result \cite[Proposition 1]{MR1288682}.

\begin{theorem}[nonlinear interpolation \cite{MR1288682}]
\label{thm:Iwaniec}
Let $(X,\mu)$ be a measure space and $E$ a Hilbert space with norm $|\cdot|$. Denote by $L^r(X,E)$ the space of $E$-valued Lebesgue integrable functions to power $r$. Let $1\leq r_1 < r < r_2$ and suppose that $T:L^r(X,E) \rightarrow L^r(X,E)$ is a bounded linear operator for all $r \in [r_1,r_2]$. If $Tw=0$ and
\[
  \frac{r}{r_2} -1 \leq \vare \leq \frac{r}{r_1}-1,
\]
then
\[
  \left\| T(|w|^\vare w) \right\|_{L^{\frac{r}{1+\vare}}(X,E)} \leq C(r;r_1,r_2) |\vare| \| w \|_{L^r(X,E)}^{1+\vare},
\]
for a constant $C(r;r_1,r_2) = \tfrac{2r(r_2-r_1)}{(r-r_1)(r_2-r)}\sup_{s \in [r_1,r_2]} \| T \|_{L^s(X,E)\to L^s(X,E)}$.
\end{theorem}

With this powerful result at hand we prove a bound on the discretely divergence free part of the Hodge decomposition of the $\bLp$ conjugate of the gradient, \ie $\bs_h = |\GRAD u_h |^{p-2} \GRAD u_h $.

\begin{corollary}[decomposition of the conjugate]
\label{cor:boundgh}
Let $1<p^-<p<p^+<\infty$. Given $u_h \in \polX_h$ define $\bs_h = |\GRAD u_h |^{p-2} \GRAD u_h \in \polQ_h$. If $\bg_h \in \polQ_h$ denotes the discretely divergence free component of the discrete Hodge decomposition of $\bs_h$ in $\bL^q(\Omega)$, then
\[
  \| \bg_h \|_{\bL^q} \leq C(p;p^-,p^+) |p-2| \| \GRAD u_h \|_{\bL^p}^{p/q},
\]
where the constant $C(p,p^-,p^+)$ is as in Theorem~\ref{thm:Iwaniec}.
\end{corollary}
\begin{proof}
We apply Theorem~\ref{thm:Iwaniec}. Set $X = \calT_h$ with $\mu(K)=|K|$ and $E=\Real^d$. Any function $\bq_h \in \polQ_h$ can be uniquely identified with a map $q: \calT_h \rightarrow E$ by $q(K) = \bq_{h|K}$. Moreover,
\[
  \| q \|_{L^r(X,E)}^r = \sum_{K \in \calT_h } |K| |q(K)|^r = \sum_{K \in \calT_h} |K| |\bq_{h|K}|^r = \sum_{K \in \calT_h} \int_K |\bq_h|^r \diff x 
  = \| \bq_h \|_{\bL^r}^r.
\]
Define the mapping $T:L^r(X,E) \rightarrow L^r(X,E)$ by $T: s \leftrightarrow \bs_h \mapsto \bg_h \leftrightarrow g$, where $\bg_h$ comes from the discrete Hodge decomposition of $\bs_h$. Theorem~\ref{thm:discHodge} then shows that this mapping is continuous, \ie
\[
  \| T s \|_{L^r(X,E)} = \| \bg_h \|_{\bL^r} \leq C \| \bs_h \|_{\bL^r} = C \| s \|_{L^r(X,E)} \quad \forall r \in (1,\infty),
\]
where the constant is uniform in any subset of $(1,\infty)$ that is bounded away from $1$.

Given $u_h \in \polX_h$, set $w\leftrightarrow \bw_h = \GRAD u_h$, so that $T(w) =0$. 
Theorem~\ref{thm:Iwaniec} with $r=p$ and $\vare = p-2$ implies the result.
\end{proof}

The second application of the discrete Hodge decomposition concerns the properties of the flux. Namely, we obtain bounds on the discretely divergence free part of $A_h \GRAD u_h$. Notice that if we denote this vector field by $\bell_h$, and we assume that $A \in L^\infty(\Omega)$, we would immediately obtain a bound of the form $\|\bell_h\|_{\bL^p} \leq \|A\|_{L^\infty}\|\GRAD u_h \|_{\bL^p}$. Since we are not assuming that our coefficients are bounded, more care is needed to bound the norm of this field, and the size of the coefficient must be replaced by the size of its oscillation. This estimate, although similar in nature, requires quite different techniques and exploits in an essential way the fact that $A$ has bounded mean oscillation.

\begin{corollary}[decomposition of the flux]
\label{col:boundlh}
Given $u_h \in \polX_h$ construct $A_h \GRAD u_h \in \polQ_h$. Denote by $\bell_h \in \polQ_h$ the discretely divergence free part of the discrete Hodge decomposition of $A_h \GRAD u_h$. Then, the following estimate holds
\[
  \| \bell_h \|_{\bL^p} \leq C |A|_\bmo \|\GRAD u_h \|_{\bL^p}
\]
for a constant $C$ independent of $h$, $A$ and $u_h$.
\end{corollary}
\begin{proof}
Theorem~\ref{thm:discHodge} readily yields that $\| \bell_h \|_{\bL^p} \leq C \| A_h \GRAD u_h \|_{\bL^p}$. We now need to estimate the norm of the flux.

Since $\bell_h$ is piecewise constant,
\begin{equation}
\label{eq:maxest}
  |\Omega| \| \bell_h \|_{\bL^p}^p = \sum_{K \in \calT_h} |K| |\GRAD u_{h|K}|^p \int_\Omega | A_{h|K}|^p \diff x,
\end{equation}
and
\[
  |A_{h|K}| = \left|\fint_K A(z) \diff z \right| \leq \calM_{\calT_h}A \leq C \calM A,
\]
where we used Lemma~\ref{lem:maximals}. Therefore, we can continue \eqref{eq:maxest} as follows
\begin{align*}
  |\Omega| \| \bell_h \|_{\bL^p}^p &\leq C \sum_{K \in \calT_h} \| \GRAD u_h \|_{\bL^p(K)}^p \int_\Omega |\calM A(x)|^p \diff x 
  \leq C \| \GRAD u_h \|_{\bL^p}^p \| \calM^\sharp A\|_{L^p}^p 
\end{align*}
where the last inequality uses the Fefferman-Stein inequality \eqref{eq:FefStein}. Recalling that
\[
  \| \calM^\sharp A\|_{L^p}^p \leq |\Omega| \| \calM^\sharp A \|_{L^\infty}^p = |\Omega| |A|_\bmo^p 
\]
yields the result.
\end{proof}

\section{Finite element discretization and convergence}
\label{sec:MKE}

This section is the core of our work. Here we will show that a finite element approximation of \eqref{eq:bvp} is convergent. We define $u_h \in \polX_h$ as the solution of
\begin{equation}
  \int_\Omega \scl A_h \GRAD u_h, \GRAD \zeta_h \scr \diff x = \int_\Omega \scl \fe, \GRAD \zeta_h \scr \diff x \quad \forall \zeta_h \in \polX_h.
\label{eq:feapprox}
\end{equation}

\subsection{A priori estimate for $p\geq 2$}
\label{sub:pgeq2}
For the time being we will assume that $p\geq 2$ and that $\fe \in \bLp$. Setting $\zeta_h = u_h$ in \eqref{eq:feapprox} immediately yields existence and uniqueness of $u_h$. It remains then to analyze its convergence properties. We begin with an a priori estimate for the gradient in $\bLp$.

\begin{theorem}[a priori estimate]
\label{thm:apriori}
Let $p \geq 2$ and $\fe \in \bLp$. There is a constant $\vare^\star = \vare^\star(|A|_\bmo)$ such that if $|p-2|<\vare^\star$ and $u_h \in \polX_h$ solves \eqref{eq:feapprox}, then
\[
  \| \GRAD u_h \|_{\bL^p} \leq C \| \fe \|_{\bL^p},
\]
where the constant $C$ depends on $|A|_\bmo$, but it is independent of $h$.
\end{theorem}
\begin{proof}
Let $\bs_h = |\GRAD u_h |^{p-2} \GRAD u_h$. Since $u_h \in \polX_h$, then $\bs_h \in \polQ_h$. Apply the discrete Hodge decomposition of Theorem~\ref{thm:discHodge} to $\bs_h$ in $\bL^q(\Omega)$, with $\tfrac1p + \tfrac1q = 1$, to obtain that there are unique $\phi_h \in \polX_h$, $\bg_h \in \polQ_h$ such that
\[
  \bs_h = \GRAD \phi_h + \bg_h, \qquad \int_\Omega \scl \bg_h, \GRAD \zeta_h \scr \diff x = 0 \quad \forall \zeta_h \in \polX_h.
\]
Set $\zeta_h = \phi_h$ in \eqref{eq:feapprox}. From Lemma~\ref{lem:approxbmo} we have that $A$ and $A_h$ have the same coercivity \eqref{eq:Apositive} constant $\alpha$, thereby obtaining
\begin{equation}
\label{eq:Lpest}
  \begin{aligned}
    \alpha \int_\Omega |\GRAD u_h |^p &\leq \int_\Omega \scl A_h \GRAD u_h, |\GRAD u_h |^{p-2} \GRAD u_h \scr \diff x \\
    &= \int_\Omega \scl \fe, \GRAD \phi_h \scr \diff x
    + \int_\Omega \scl A_h \GRAD u_h, \bg_h \scr \diff x.
  \end{aligned}
\end{equation}

We estimate each term on the right hand side of this inequality separately. Using the estimates of Theorem~\ref{thm:discHodge} yields
\[
  \int_\Omega \scl \fe, \GRAD \phi_h \scr \diff x \leq \| \fe \|_{\bL^p} \| \GRAD \phi_h \|_{\bL^q} \leq C \| \fe \|_{\bL^p} \| \bs_h \|_{\bL^q}.
\]
Using the relation between $p$ and $q$ and the definition of $\bs_h$ we obtain
\[
  \int_\Omega |\bs_h|^q \diff x = \int_\Omega \left| |\GRAD u_h|^{p-2} \GRAD u_h \right|^q \diff x = \int_\Omega |\GRAD u_h |^p \diff x,
\]
which allows us to conclude
\[
  \int_\Omega \scl \fe, \GRAD \phi_h \scr \diff x \leq C \| \fe \|_{\bL^p} \|\GRAD u_h \|_{\bL^p}^{p/q}.
\]

To estimate the second term in \eqref{eq:Lpest} we apply the discrete Hodge decomposition of Theorem~\ref{thm:discHodge} to $A_h \GRAD u_h \in \polQ_h$ in $\bLp$. In doing so we obtain $A_h \GRAD u_h = \GRAD \psi_h + \bell_h$ with $\psi_h \in \polX_h$ and $\bell_h$ discretely divergence free. This decomposition also yields
\[
  \int_\Omega \scl A_h \GRAD u_h, \bg_h \scr \diff x = \int_\Omega \scl \bell_h , \bg_h \scr \diff x \leq \| \bell_h \|_{\bL^p} \| \bg_h \|_{\bL^q}.
\]
To proceed we must invoke the estimates of Corollary~\ref{cor:boundgh} and Corollary~\ref{col:boundlh}, so that
\[
  \int_\Omega \scl A_h \GRAD u_h, \bg_h \scr \diff x \leq \| \bell_h \|_{\bL^p} \| \bg_h \|_{\bL^q}
  \leq C |p-2| |A|_\bmo \| \GRAD u_h \|_{\bL^p}^{p}.
\]
Since $1+\tfrac{p}{q}=p$. Notice that because $p\geq 2$, the constant $C(p,p^-,p^+)$ of Corollary~\ref{cor:boundgh} can be chosen uniform in $p$.

These estimates allow us to rewrite \eqref{eq:Lpest} as 
\[
  \| \GRAD u_h \|_{\bL^p}^p \leq C \left( \| \fe \|_{\bL^p} \| \GRAD u_h \|_{\bL^p}^{p-1} + |p-2||A|_\bmo \| \GRAD u_h \|_{\bL^p}^p \right).
\]
Setting $|p-2|$ sufficiently small the coefficient of the second term on the right hand side can be made less than one (this will depend on $|A|_\bmo$). This yields the result.
\end{proof}

\begin{remark}[range of $p$]
\label{rem:rangep}
The conclusion of Theorem~\ref{thm:apriori} is similar in nature to those of \cite{MR1837269,MR2110431} in the sense that we have \eqref{eq:feapprox} well posed over a range of $p$ that depends on the size of the oscillation of the coefficient. This bound seems natural because the size of the coefficients is not involved (one can multiply the equation by any constant) but rather the size of their oscillation, which is the generalization of the ratio $\alpha/M$ between the largest and smallest eigenvalues for a bounded coefficient and the quantity the $\bLp$ estimates on the gradient of Meyers \cite{MR0159110} depend on.
\end{remark}

\begin{remark}[the need for a nonstandard approach]
\label{rem:LM}
At the beginning of this section we established existence and uniqueness of $u_h$ by setting the test function $\zeta_h = u_h$. Another way of realizing this is via a standard application of the Lax Milgram lemma. To do this it was imperative to approximate the coefficient $A$ by piecewise constants $A_h$, so that the associated bilinear form is bounded. This bound, however, cannot be uniform in $h$. This shows that even for the case $p=2$ the a priori estimate of Theorem~\ref{thm:apriori} is nontrivial.
\end{remark}

\subsection{Convergence for $p\geq2$}
\label{sub:convpgeq2}

The a priori estimate obtained in Theorem~\ref{thm:apriori} shows that the family of discrete solutions $\{u_h\}_{h>0}$ is uniformly bounded in $W^1_0(\Lp)$. Consequently there is a sequence $\{u_{h_k}\}_{k \in \polN} \subset \{u_h\}_{h>0}$ and $\hu \in W^1_0(\Lp)$ such that $\GRAD u_{h_k} \rightharpoonup \GRAD \hu$ in $\bLp$ and $u_{h_k} \to \hu$ in $\Lp$. Moreover, weak lower semicontinuity of the norm yields
\[
  \| \GRAD \hu \|_{\bL^p} \leq C \| \fe \|_{\bL^p}.
\]
It remains then to show that this limit solves \eqref{eq:bvp}. From uniqueness it will follow that $\hu = \ue$.

\begin{proposition}[weak convergence]
\label{prop:wconv}
There is a sequence $\{u_{h_k}\}_{k \in \polN} \subset \{u_h\}_{h>0}$ and $\hu \in W^1_0(\Lp)$ such that, as $k \to \infty$, the sequence converges weakly in $W^1_0(\Lp)$ and strongly in $\Lp$ to $\hu$. Moreover, the function $\hu \in W^1_0(\Lp)$ solves \eqref{eq:bvp}.
\end{proposition}
\begin{proof}
The convergence is an immediate corollary of the a priori estimate of Theorem~\ref{thm:apriori}.
To show that $\hu$ solves \eqref{eq:bvp} let $\varphi \in C^\infty_0(\Omega)$. One can construct $\varphi_h \in \polX_h$ such that $\varphi_h \to \varphi$ in $W^1(L^\infty(\Omega))$ as $h \to 0$. Taking the Lagrange interpolant suffices for these purposes \cite{MR0520174,MR2050138}. Thus,
\[
  \int_\Omega \scl A\GRAD \hu, \GRAD \varphi \scr \diff x = \int_\Omega \scl A \GRAD \hu, \GRAD \varphi_h \scr \diff x 
  + \int_\Omega \scl A \GRAD \hu, \GRAD (\varphi - \varphi_h ) \scr \diff x.
\]
As $h \to 0$ the second term can be estimated as
\begin{align*}
  \int_\Omega \scl A \GRAD \hu, \GRAD (\varphi - \varphi_h ) \scr \diff x &\leq \| A \|_{L^q} \| \GRAD \hu \|_{\bL^p} 
  \| \GRAD (\varphi - \varphi_h ) \|_{\bL^\infty} \\ 
  &\leq 
  C \| A \|_{L^q} \| \fe \|_{\bL^p} \| \GRAD (\varphi - \varphi_h ) \|_{\bL^\infty} \to 0,
\end{align*}
where we used the a priori estimate of Theorem~\ref{thm:apriori} and $A \in L^q(\Omega)$.

It remains to deal with the first term
\[
  \int_\Omega \scl A \GRAD \hu, \GRAD \varphi_h \scr \diff x = \int_\Omega \scl A_h \GRAD u_h, \GRAD \varphi_h \scr \diff x
  + \int_\Omega \scl A \GRAD \hu - A_h \GRAD u_h, \GRAD \varphi_h \scr \diff x = \calS + \calR.
\]
Evidently, $\calS = \int_\Omega \scl \fe, \GRAD \varphi_h \scr \diff x$ and, as $h \to 0$, $\calS \to \int_\Omega \scl \fe, \GRAD \varphi \scr \diff x$, so that if we show that $\calR \to 0$ we obtain the result. Notice that
\[
  \calR = \int_\Omega \scl (A - A_h ) \GRAD\hu, \GRAD \varphi_h \scr \diff x + \int_\Omega \scl A_h \GRAD(\hu - u_h ), \GRAD \varphi_h \scr \diff x,
\]
so that, using the results of Lemma~\ref{lem:approxbmo}, we obtain
\[
  \int_\Omega \scl (A - A_h ) \GRAD\hu, \GRAD \varphi_h \scr \diff x \leq \| A - A_h \|_{L^q} \| \GRAD \hu \|_{\bL^p} \| \GRAD \varphi_h \|_{\bL^\infty}
  \to 0
\]
since $\|\GRAD \varphi_h \|_{\bL^\infty} \leq C \| \GRAD \varphi\|_{\bL^\infty}$. For the second term we have that $A_h\tr \GRAD \varphi_h \to A\tr \GRAD \varphi$ in $\bL^q(\Omega)$ and, by passing to a subsequence, $\GRAD u_{h_k} \rightharpoonup \GRAD \hu$ in $\bLp$, hence
\[
  \int_\Omega \scl A_{h_k} \GRAD(\hu - u_{h_k} ), \GRAD \varphi_h \scr \diff x 
    = \int_\Omega \scl \GRAD(\hu - u_{h_k} ), A_{h_k}\tr \GRAD \varphi_{h_k} \scr \diff x \to 0.
\]

In conclusion $\hu$ is a solution of \eqref{eq:bvp}.
\end{proof}

The previous result shows that the weak limit $\hu$ is a solution. Let us now show that strong convergence does take place but in a weaker norm.

\begin{proposition}[strong convergence]
\label{prop:sconv}
For any $\hp \in [2,p)$ we have $u_h \to \hu$ in $W^1_0(L^\hp(\Omega))$.
\end{proposition}
\begin{proof}
Let $\hp \in [2,p)$. Since $\fe \in \bLp \subset \bL^{\hp}(\Omega)$, problem \eqref{eq:feapprox} is also well posed in $W^1_0(L^\hp(\Omega))$. As it is well known \cite{MR2050138,MR2648380}, this is equivalent to an inf-sup condition
\[
  \| \GRAD w_h \|_{\bL^{\hp}} \leq C \sup_{0 \neq \zeta_h \in \polX_h} 
    \frac{ \int_\Omega \scl A_h \GRAD w_h , \GRAD \zeta_h \scr \diff x} { \| \GRAD \zeta_h \|_{\bL^\hq} },
\]
where $\tfrac1\hp + \tfrac1\hq = 1$ and the constant $C$ is independent of $h$.

Since $\hu \in W^1_0(\Lp)$ there is a family $\{ w_h \in \polX_h \}_{h>0}$ such that $w_h \to \hu$ in $W^1_0(\Lp)$ (the Cl\'ement interpolant suffices \cite{MR2373954}). Consequently,
\[
  \| \GRAD (\hu - u_h) \|_{\bL^\hp} \leq \| \GRAD (\hu - w_h) \|_{\bL^\hp} + \| \GRAD (w_h - u_h) \|_{\bL^\hp}.
\]
By construction, $\| \GRAD (\hu - w_h) \|_{\bL^\hp} \to 0 $ as $h\to 0$. For the second term we use the discrete inf-sup
\[
  \| \GRAD (w_h - u_h) \|_{\bL^\hp} \leq
  C \sup_{0 \neq \zeta_h \in \polX_h} \frac{ \int_\Omega \scl A_h \GRAD(w_h - u_h ), \GRAD \zeta_h \scr \diff x}{\| \GRAD \zeta_h \|_{\bL^\hq}},
\]
from which we conclude
\begin{align*}
  \int_\Omega \scl A_h \GRAD(w_h - u_h ), \GRAD \zeta_h \scr \diff x &=
  \int_\Omega \scl A_h \GRAD w_h , \GRAD \zeta_h \scr \diff x - \int_\Omega \scl \fe, \GRAD \zeta_h \scr \diff x \\
  &= \int_\Omega \scl A_h \GRAD w_h - A \GRAD \hu, \GRAD \zeta_h \scr \diff x \\
  &= \int_\Omega \scl A_h (\GRAD w_h -\GRAD \hu), \GRAD \zeta_h \scr \diff x \\
  &+ \int_\Omega \scl (A_h - A) \GRAD \hu, \GRAD \zeta_h \scr \diff x
\end{align*}
\ie
\begin{multline}
  \int_\Omega \scl A_h \GRAD(w_h - u_h ), \GRAD \zeta_h \scr \diff x =
  \int_\Omega \scl A_h (\GRAD w_h -\GRAD \hu), \GRAD \zeta_h \scr \diff x \\
  + \int_\Omega \scl (A_h - A) \GRAD \hu, \GRAD \zeta_h \scr \diff x.
\label{eq:coeffapprox}
\end{multline}

The first term in \eqref{eq:coeffapprox} can be treated as follows
\[
  \sup_{0 \neq \zeta_h \in \polX_h} \frac{\int_\Omega \scl A_h (\GRAD w_h -\GRAD \hu), \GRAD \zeta_h \scr \diff x}{\| \GRAD \zeta_h \|_{\bL^\hq}}
  \leq \| A_h \|_{L^r} \|\GRAD(w_h - \hu ) \|_{\bL^p} \to 0,
\]
where, since $\hp < p$, $r = (\tfrac1\hp - \tfrac1p)^{-1}>1$ and we used that, by Lemma~\ref{lem:approxbmo}, $A_h$ is bounded in $L^r(\Omega)$ and that $w_h \to \hu$ in $W^1_0(\Lp)$. The second term of \eqref{eq:coeffapprox} is treated similarly
\[
  \sup_{0 \neq \zeta_h \in \polX_h} \frac{\int_\Omega \scl (A_h - A)\GRAD \hu, \GRAD \zeta_h \scr \diff x}{\| \GRAD \zeta_h \|_{\bL^\hq}} \leq 
  \| A_h - A \|_{L^r} \| \GRAD \hu \|_{\bL^p} \to 0.
\]

Collecting the obtained estimates yields the result.
\end{proof}

\subsection{The case $p \in (1,2)$}
\label{sub:ple2}

Let us briefly comment on the case $p \in (1,2)$. In this case we cannot set $\zeta_h = u_h$ in \eqref{eq:feapprox} to obtain existence and uniqueness, since $\fe \not\in \Ldeuxd$. However, owing to the symmetry of $A_h$, we can resort to inf-sup theory \cite{MR2648380} to assert the existence, uniqueness and a priori estimate of a solution. The a priori estimate yields a weakly convergent sequence whose limit can also be shown to be a solution to the problem \eqref{eq:bvp}. This is made precise in the following result.

\begin{theorem}[convergence for $p<2$]
Let $1<p<2$ and $\fe \in \bLp$. There is a constant $\vare^\star$ that depends only on the oscillation of $A$, such that if $|p-2|<\vare^\star$ and $u_h \in \polX_h$ solves \eqref{eq:feapprox}, then for a constant $C=C(|A|_\bmo)$, but independent of $h$
\[
  \| \GRAD u_h \|_{\bL^p} \leq C \| \fe \|_{\bL^p}.
\]
The family $\{ u_h \in \polX_h\}_{h>0}$ converges, as $h \to 0$, weakly in $W^1_0(\Lp)$ and strongly in $\Lp$ to $\hu$, which is a solution to \eqref{eq:bvp}. Finally, for every $\hp \in (2-\vare^\star, p)$ we have $u_h \to \hu$ in $W^1_0(L^\hp(\Omega))$.
\end{theorem}
\begin{proof}
The proof of existence, uniqueness and a priori estimate is a standard application of the inf-sup theory \cite{MR2648380}. The weak convergence repeats the arguments of Proposition~\ref{prop:wconv} and the strong convergence those of Proposition~\ref{prop:sconv}. For brevity we skip the details.
\end{proof}

\section{About the possibility of establishing convergence rates}
\label{sec:convrates}

Let us comment on the possibility of establishing convergence rates for $u_h$, at least in $W^1_0(L^\hp(\Omega))$, where strong convergence takes place. This is useful to understand the critical nature of the $\BMO$ coefficients in the theory and approximation of elliptic equations. To make the discussion simple and focus on the essential difficulties we will assume in this section, and this section only, that the family of meshes $\{\calT_h\}_{h>0}$ is quasiuniform so that $h$ can be identified with the mesh size.

The heart of the matter lies in the proof of Proposition~\ref{prop:sconv}. The steps are rather standard, first the error is split into the interpolation and approximation errors
\[
  \| \GRAD ( \hu - u_h ) \|_{\bL^\hp} \leq \| \GRAD (\hu - w_h ) \|_{\bL^\hp} + \| \GRAD ( w_h - u_h ) \|_{\bL^\hp},
\]
for a suitably chosen $w_h \in \polX_h$. Next we must bound each one of these terms. The first one --- the interpolation error --- is not at all related to the numerical scheme but rather to the smoothness of the solution and the approximation properties of the finite element space $\polX_h$. For this reason, if one were to assume that $\hu \in W^{1+s}(L^\hp(\Omega)) \cap W^1_0(L^\hp(\Omega))$ for some $s>0$, we can conclude that
\[
  \| \GRAD (\hu - w_h ) \|_{\bL^\hp} \leq C h^s,
\]
for a constant independent of $h$. Although this is standard, we must reiterate that such an estimate is independent of the problem in question and a similar smoothness assumption on the solution $\hu = \ue$ must be made (or proved) if rates of convergence are desired in any other problem, say one with bounded or even smooth coefficients.

The second term --- the approximation error --- on the other hand encodes how well our numerical scheme reproduces the exact problem and, in our case, is of more interest. The proof of Proposition~\ref{prop:sconv} yields
\[
  \| \GRAD ( w_h - u_h ) \|_{\bL^\hp} \leq C \left(
    \| A \|_{L^r} \| \GRAD( w_h - \hu ) \|_{\bL^p} + \| A - A_h \|_{L^r} \| \GRAD \hu \|_{\bL^p}
  \right),
\]
for some $r>1$. The first term on the right hand side of this inequality can be treated with a similar argument as the interpolation error. The second term is an unavoidable consistency error (see Remark~\ref{rem:LM}) and to conclude a rate of decay for it one must study how well the coefficient $A$ can be approximated by a piecewise constant function.

The so-called Campanato spaces $\calL^\lambda(L^r(\Omega))$ \cite[Chapter 4]{MR0482102} provide the right tool to quantify the rate of approximation by piecewise constants. They are defined as
\[
  \calL^\lambda(L^r(\Omega)) = \left\{ w \in L^r(\Omega): [w]_{\lambda,r} < \infty \right\}
\]
with semi-norm
\[
  [w]_{\lambda,r}^r = \sup_{x \in \Omega, \ \rho>0}
    \rho^{-\lambda} \int_{B(x,\rho)\cap \Omega} | w - w_{B(x,\rho)}|^r \diff z,
\]
where $B(x,\rho)$ denotes the ball of radius $\rho$ centered at $x$. It is well known \cite[\S4.7]{MR0482102} that $\calL^d(L^r(\Omega)) = \BMO$ for any $r \in [1,\infty)$ and that \cite[Theorem~4.6.1]{MR0482102} if $d < \lambda \leq d+r$ we have $\calL^\lambda(L^r(\Omega)) = C^{0,\beta}(\bar\Omega)$ for $\beta = \tfrac{\lambda - d}r$. Moreover, $\calL^\lambda(L^r(\Omega))$ with $\lambda > d + r$ contains only constant functions.

With this functional setting at hand we see that the only plausible way to assert a rate of convergence for $A_h$ would be as follows:
\begin{align*}
  \| A - A_h \|_{L^r}^r &=
    \sum_{K \in \calT_h } \int_K |A - A_h|^r \diff z
    = \sum_{K \in \calT_h } \int_K |A - A_h|^r \diff z \frac{h^{d+\delta}}{h^{d+\delta}} \\
  & \leq C [A]_{d+\delta,r}^r \sum_{K \in \calT_h} h^{d+\delta}
    \leq C [A]_{d+\delta,r}^r |\Omega| h^\delta,
\end{align*}
for some $\delta > 0$. However, as the aforementioned embeddings show, this already implies that $A \in C^{0,\beta}(\bar\Omega)$ for some $\beta>0$ and we go back to the classical case of smooth coefficients. A similar assumption, with analogous consequences would be to assume that the coefficients $A$ lie in a slightly better space than $W^1(L^d(\Omega))$, see \cite{MR1005611} and \cite[Theorem 5.22]{MR2777537}.

In summary, the assumption $A \in \BMO$ is critical. Weaker assumptions render us unable to even assert existence of solutions while even the slightest stronger assumption brings us back to the classical case, where standard and well known techniques apply.

\bibliographystyle{plain}
\bibliography{biblio}

\end{document}